\documentclass[a4paper]{article}
\AtBeginDocument{%
  \providecommand\BibTeX{{%
    \normalfont B\kern-0.5em{\scshape i\kern-0.25em b}\kern-0.8em\TeX}}}

\usepackage{enumitem}
\usepackage[utf8]{inputenc}
\usepackage{mathtools}
\usepackage{amsfonts}
\usepackage{amsthm}
\usepackage{amsmath}
\usepackage{amssymb}
\usepackage{tikz}
\usetikzlibrary{trees}
\usepackage{hyperref}
\usepackage{enumitem}

\newtheorem{Problem}{Problem}

\newtheorem{Theorem}{Theorem}
\newtheorem{Proposition}{Proposition}
\newtheorem{Lemma}{Lemma}

\newtheorem{Corollary}{Corollary}

\newtheorem{Definition}{Definition}
\newtheorem{Example}{Example}

\newcommand{\uproman}[1]{\uppercase\expandafter{\romannumeral#1}}

\usepackage{pgfplots}\pgfplotsset{compat=1.13}
\usetikzlibrary{arrows,positioning,calc,patterns}

\def\eatspace#1{#1}
\def\step#1#2{\par\kern1pt\dimen44=#2em\advance\dimen44 1.67em\hangindent\dimen44\hangafter=1\noindent\rlap{\small#1}\kern\dimen44\relax\eatspace}

\def\<#1>{\langle#1\rangle}

\AtBeginDocument{%
  \providecommand\BibTeX{{%
    Bib\TeX}}}

\begin{document}

\title{On finite orbits of infinite correspondences}

\author{Manfred Buchacher}

\maketitle

\section{Abstract}

These notes collect results about algebraic correspondences and adapt them to the setting of correspondences on projective lines. The focus lies on finite orbits of algebraic correspondences. The main result is a field theoretic characterization of the (in)finiteness of the number of finite orbits.

\section{Introduction}

An algebraic correspondence between two projective lines given by a polynomial~$p(X,Y)$ can be interpreted as a multivalued map that assigns to a point $z$ on the projective line the roots of $p(z,Y)$ thereon. Correspondences naturally arise in various settings, the common object of study being properties of a point under iterative application of a correspondence. They generalize rational self-maps~\cite{silverman2007arithmetic} and a considerable part of the literature about them deals with extending important notions for such maps (e.g. canonical heights, its properties under iteration~\cite{ingram2019canonical}, iteration and its behavior under Galois actions~\cite{ingram2018p}). Questions on the structure of finite sets that are invariant under a correspondence have been addressed in~\cite{bell2022invariant, bellaiche2023self, krishnamoorthy2018correspondences, ingram2017critical}. However, important problems remain unsolved~\cite[Prob~2,~3,~4]{buchacher2024separated}.  
\bigskip

The present notes originate from the investigation of field theoretic problems that was started, continued and further extended in~\cite{buchacher2020separating, buchacher2024separating} and~\cite{buchacher2024separated}, respectively. These problems turn out to be intimately connected to questions on algebraic correspondences~\cite[Sec~6]{buchacher2024separating},~\cite[Sec~6]{buchacher2024separated}.   
%As a consequence their thorough study cannot be avoided.

%Among these are questions about computing intersections of fields~\cite{fried1978poncelet, buchacher2024separating, hardouin} or finding decompositions of elements thereof~\cite{buchacher2024separated, hardouin}. The algorithmic nature of these problems raises computational questions for algebraic correspondences. A list of such questions is presented in~\cite[Sec~6]{buchacher2024separated}. The purpose of these notes is to answer one of them.
\bigskip

%\cite[Remark~5.1]{hardouin2021differentially}, 
%\cite{gotou2023dynamical}, \cite{dinh2020dynamics}

The paper is organized as follows. In Section~\ref{sec:prelim} we recall the most basic definitions related to algebraic correspondences  and give a field-theoretic characterization of their finiteness as well as some examples that illustrate different phenomena. In Section~\ref{sec:main} we show that any infinite correspondence has only finitely many finite orbits. The arguments are of algebraic-geometric nature and we recall the most important notions from algebraic geometry to keep the paper self-contained. The paper closes with Section~\ref{sec:prob} and the open problem on how big a finite orbit can be.

\section{Preliminaries}\label{sec:prelim}

We assume throughout that $k$ is an algebraically closed field of characteristic zero. We denote by $k[X,Y]$ the ring of polynomials in $X$ and $Y$ over $k$, and we write $k(X,Y)$ for its quotient field. Let $p$ be an irreducible polynomial in $k[X,Y]$ that is neither an element of $k[X]$ nor of $k[Y]$. We denote by $C$ the curve in $\mathbb{P}^1\times \mathbb{P}^1$ that is associated with its bi-homogenization 
\begin{equation*}
x_0^{\deg_x p}y_0^{\deg_y p}p\left(x_1/x_0,y_1/y_0\right).
\end{equation*}
We consider $C$ as an \textbf{algebraic correspondence} between two projective lines, that is, a binary relation over two copies of $\mathbb{P}^1$ that is an algebraic variety. We denote by $\pi_1$ and $\pi_2$ the projections of $C$ on the first and second coordinate, respectively. Given $z\in\mathbb{P}^1$, we define 
\begin{equation*}
C(z) := \pi_2\left( \pi_1^{-1}\left(z\right) \right) \qquad \text{and} \qquad ^{\mathrm{tr}}C(z) :=  \pi_1\left( \pi_2^{-1}\left(z\right) \right).
\end{equation*}
The correspondence is said to be \textbf{finite} if there is a positive integer $n$ such that 
\begin{equation}\label{eq:period}
(^{\mathrm{tr}}C \circ C)^{(n)} \equiv (^{\mathrm{tr}}C \circ C)^{(n+1)}.
\end{equation}
The smallest $n$ such that~\eqref{eq:period} holds is called the \textbf{period} of the correspondence. 

Let $\sim$ be the smallest equivalence relation on $C$ such that
\begin{equation*}
(z_1,z_2) \sim (w_1,w_2) \quad \text{whenever} \quad z_1 = w_1 \quad \text{or} \quad z_2 = w_2. 
\end{equation*}
The equivalence class of $(z_1,z_2)$ is referred to as the \textbf{orbit} of $(z_1,z_2)$. 
\bigskip

Let $x$ and $y$ be the equivalence classes of $X$ and $Y$ in $k[X,Y] / \langle p(X,Y) \rangle$, and let $k(x,y)$ be its quotient field. The following theorem gives a field-theoretic \textbf{characterization} of finite correspondences. We refer to~\cite[Lem~1, Thm~1]{fried1978poncelet} and~\cite[Thm~4.3]{hardouin} for proofs.

\begin{Theorem}\label{thm:Fried}
Assume that $C$ is normal. Then the algebraic correspondence associated with $C$ is finite if and only if 
\begin{equation*}
k(x) \cap k(y) \neq k.
\end{equation*}
\end{Theorem}

We present two examples of a correspondence and illustrate Theorem~\ref{thm:Fried}.

\begin{Example}
Let us consider the algebraic correspondence defined by 
\begin{equation*}
p = XY - X^2Y - XY^2 - 1.
\end{equation*}
Since 
\begin{equation*}
\frac{X-Y}{XY} p = \frac{1 + X^2 - X^3}{X} - \frac{1 + Y^2 - Y^3}{Y},
\end{equation*}
and hence 
\begin{equation*}
k(x) \cap k(y) = k\left(\frac{1+x^2-x^3}{x}\right),
\end{equation*}
it follows that the correspondence is finite. In particular, so is any of its orbits.
\end{Example}

\begin{Example}
The correspondence associated with 
\begin{equation*}
p = X^2 + 3XY + Y^2
\end{equation*}
is not finite, and there are different ways so see this. Since 
\begin{equation*}
p = 0 \quad \Longleftrightarrow \quad Y = \frac{-3 \pm \sqrt{5}}{2} X,
\end{equation*}
and 
\begin{equation*}
\left(\frac{-3+\sqrt{5}}{2}\right) \left(\frac{-3-\sqrt{5}}{2}\right) = 1
\end{equation*}
it easily follows that any orbit is of the form 
\begin{equation*}
\left\{ 
\lambda \cdot
\left( 
\left(\frac{-3 + \sqrt{5}}{2}\right)^k, \left(\frac{-3 + \sqrt{5}}{2}\right)^{k\pm 1}
\right) : k \in\mathbb{Z}
\right\}.
\end{equation*}
for some $\lambda\in\mathbb{P}^1$. Obviously, an orbit is not finite unless $\lambda \in\{ 0, \infty \}$. So the correspondence is not finite either. 

Alternatively, one could argue that if the correspondence were finite, and hence $k(x)\cap k(y) \neq k$, then $p$ had a multiple of the form $aX^n - bY^n$, with $a,b\in k\setminus\{0\}$ and $n\in \mathbb{N}$~\cite[Sec~4]{buchacher2024separating}. But then $p(X,1) \mid aX^n - b$, and the quotient of any pair of roots of $p(X,1)$ were a root of unity. However, the roots of $p(X,1)$ are $(-3+\sqrt{5})/2$ and $(-3-\sqrt{5})/2$. So $k(x)\cap k(y) = k$, and the correspondence is infinite.
\end{Example}

In the previous example all but finitely many orbits were infinite. The purpose of the next section is to show that this is not a coincidence but true for any infinite correspondence.

\section{Finite orbits}\label{sec:main}

Any orbit of a finite correspondence is finite, and there is an upper bound on the size of a finite orbit that holds uniformly. The following proposition provides a (rough) bound.

\begin{Proposition}
If $C$ is a correspondence of period $n$, then all its orbits are finite, and the cardinality of any orbit is bounded by $\max\{\deg_x p, \deg_y p\}^{2(n+1)}$. 
\end{Proposition}
\begin{proof}
The statement follows from the observation that for any $z\in C$ the cardinality of $\pi_1^{-1}(\pi_1(z))$ and $\pi_2^{-1}(\pi_2(z))$ is at most $\deg_y p$ and $\deg_x p$, respectively. 
\end{proof}

The purpose of these notes is to explain that there are only finitely many finite orbits when the correspondence is not finite. 

\begin{Theorem}\label{thm:main}
The number of finite orbits of a normal infinite correspondence is finite. Its number can be bounded as a function of the degree of the correspondence.
\end{Theorem}

Proofs of Theorem~\ref{thm:main} can be found in~\cite[Thm~1.1]{bell2022invariant} and~\cite[Thm~2.2.1]{bellaiche2023self} in a much more general setting. However, the arguments appear to be much simpler for algebraic correspondences on the projective line. We present them here in the hope that the reader appreciates their simplicity.

We outline the strategy for establishing Theorem~\ref{thm:main}. Any finite orbit that avoids a certain finite set of points can be associated with a rational function of the form $f(x) / g(y)$ that does not have any finite poles on $C$. If there are sufficiently many finite orbits, then the corresponding rational functions can be combined to a function $f_0(x) / g_0(y)$ that does not have any poles on $C$ at all. Thus, it is some (non-zero) constant $c$. Hence $f_0(x) = c g_0(y)$, and thus $k(x) \cap k(y) \neq k$. A contradiction to Theorem~\ref{thm:Fried}.      
\bigskip

In order to work out the details we make use of some concepts and results from algebraic geometry. We recall some of them below. For details, however, we refer to~\cite{fulton2008algebraic}, in particular to Section~$2$ and Section~$3$ therein.

\begin{Definition}
Let $\mathcal{O}$ be a finite orbit on $C$ that avoids the points at infinity, and let $A$ and $B$, respectively, be its projections on the first and second coordinate. We define by
\begin{equation*}
\Theta:= \left(\prod_{\lambda\in A} (x-\lambda)\right)/\left(\prod_{\mu\in B} (y-\mu)\right)
\end{equation*}
the rational function on $C$ associated with $\mathcal{O}$.
\end{Definition}

In general, $\Theta$ might have finite roots and poles on $C$. However, it does not, if $\mathcal{O}$ avoids a certain set of points. 
\bigskip

Since $p$ is irreducible, and hence relatively prime to $\partial p/ \partial X$ and $\partial p / \partial Y$, the elimination ideals of $\langle p, \partial p / \partial X \rangle$ and $\langle p, \partial p / \partial Y \rangle$ are non-trivial. Let $a(X)$, $b(Y)$ and $c(X)$, $d(Y)$, respectively, be their (non-zero) generators, and let $\Omega$ be the finite set of points on $C$ whose first coordinate is a root of $a(X)c(X)$ or whose second coordinate is a root of $b(Y)d(Y)$. We claim that the following lemma holds.

\begin{Lemma}\label{lem:lem1}
If $\mathcal{O}$ avoids $\Omega$, then $\Theta$ does not have any finite roots and poles.
\end{Lemma}
In order to prove the lemma, we make use of the fact that any non-singular point $P$ on $C$ gives rise to a \textbf{valuation} $\nu_P$ on $k(x,y)$. We recall that the non-units of
\begin{equation*}
\mathcal{O}_P(C):=\{ f\in k(x,y) : f(P) \in k \},
\end{equation*}
constitute a unique maximal ideal
\begin{equation*}
\mathfrak{m}_P(C) := \{ f\in \mathcal{O}_P(C) : f(P) = 0 \}.
\end{equation*}
If $P$ is non-singular, then $\mathfrak{m}_P(C)$ is principal, and any line $L$ through $P$ that is not tangent to $C$ at $P$ gives rise to a generator $l$. Any non-zero $z\in\mathcal{O}_P(C)$ may then be uniquely written as $z = u l^n$, $u$ a unit in $\mathcal{O}_P(C)$, $n$ a non-negative integer~\cite[Chap~3, Thm~1]{fulton2008algebraic}. The exponent $n$ is called the valuation of $z$ and denoted by $\nu_P(z)$. The valuation of $0$ is defined by $\nu_P(0) := \infty$. It naturally extends to $k(x,y)$, the quotient field of $\mathcal{O}_P(C)$, by $\nu_P(u_1 l^n / u_2 l^m) := n - m$, $u_1$ and $u_2$ units in $\mathcal{O}_P(C)$, $n_1$ and $n_2$ non-negative integers.  
\bigskip

We now come to the proof of Lemma~\ref{lem:lem1}.
\begin{proof}
If $P$ is a finite point on $C$ that is not an element of $A\times B$, then $\Theta(P)$ is non-zero and finite. So let us assume that $P$ is a point on $C$ that is an element of $A\times B$. The tangent of $C$ at $P = (\alpha, \beta)$ is defined by
\begin{equation*}
(X-\alpha) \frac{\partial p(\alpha,\beta)}{\partial X} + (Y-\beta) \frac{\partial p(\alpha,\beta)}{\partial Y} = 0
\end{equation*}
Since $\mathcal{O}$ avoids $\Omega$, we have
\begin{equation*}
\frac{\partial p}{\partial X}(\alpha,\beta) \neq 0 \quad \text{and} \quad \frac{\partial p}{\partial Y}(\alpha,\beta) \neq 0,
\end{equation*}
so
\begin{equation*}
\nu_P(x-\alpha) = 1 = \nu_P(y-\beta).
\end{equation*}
Hence 
\begin{equation*}
\nu_P(\Theta) = 0.
\end{equation*} 
It follows that all roots and poles of $\Theta$ are supported at the points at infinity. 
\end{proof}

The following corollary shows that finite orbits avoiding $\Omega$ relate to the unit group of the coordinate ring of $p$.

\begin{Corollary}
If $\mathcal{O}$ avoids $\Omega$, then $\Theta$ and $1/\Theta$ are regular on the locus of the affine curve defined by $p$. Hence they define units in the coordinate ring of $p$.
\end{Corollary}
\begin{proof}
In the proof of Lemma~\ref{lem:lem1} we showed that $\Theta$ does not have any finite roots and poles. So both $\Theta$ and $1/\Theta$ are regular on the affine curve, and hence define units in the coordinate ring of $p$~\cite[Chap~1 Sec~3 Thm~4]{shafarevich1994basic}.
\end{proof}

We now come to the proof of Theorem~\ref{thm:main}.

\begin{proof} 
We first show that if there were too many finite orbits, then the correspondence would be finite. The number of points at infinity is finite, and so is the number of elements of $\Omega$. Hence if there are sufficiently many finite orbits, then there are also finite orbits that avoid these points. Each such finite orbit gives rise to a rational function $\Theta$ whose zeros and poles are at infinity. If there were more such orbits than points at infinity, then by looking at their divisors, we could construct a function $\Theta_1^{n_1}\cdots \Theta_k^{n_k}$ that does not have any poles at all. Hence it must be in $k\setminus \{0\}$~\cite[Chap~1 Sec~5 Cor~1]{shafarevich1994basic}. Since everything in the group generated by the $\Theta_i$'s is a ratio of rational functions in $x$ and $y$, respectively, this then says that we have a relation $f(x) = g(y)$ in $k(x,y)$. So $k(x) \cap k(y)\neq k$, and the first part of the statement follows from Theorem~\ref{thm:Fried}. To prove the second part we just note that the number of points at infinity is bounded by the degrees of $p$. Since $a(X)$, $b(Y)$ and $c(X)$, $d(Y)$ are divisors of the resultants of $p$ and $\partial p / \partial X$ and $p$ and $\partial p / \partial Y$, respectively, and the degrees of the resultants can be bounded as a function of the degree of $p$, the same is true for cardinality of $\Omega$. Thus we're done. 
\end{proof}

\section{Open problems}\label{sec:prob}

Any orbit of a finite correspondence is finite and there is a bound on their size that is uniform. An orbit of an infinite correspondence may be finite, however, there are only finitely many that are not infinite. Hence there is a uniform bound on the size of a finite orbit when a correspondence is infinite. It is therefore natural to ask how such an upper bound could be determined.

\begin{Problem}
Given $p(X,Y)\in k[X,Y]\setminus\left(k[X]\cup k[Y]\right)$, determine a uniform upper bound on the size of a finite orbit. 
\end{Problem}

There are answers on special instances of the problem, see~\cite[Remark~5.1]{hardouin2021differentially}. However, in general it is an open problem.

%For instance, for correspondences of bi-degree~$2$, it is explained in~\cite[Remark~5.1]{hardouin2021differentially} that the size of a finite orbit cannot exceed~$23$. 

\section{Acknowledgements}

Thanks go to the Johannes Kepler University Linz and the state of Upper Austria which supported this work with the grant LIT-2022-11-YOU-214. Many thanks go to Jason Bell for the many discussions which preceded these notes.   

\bibliographystyle{plain}
\bibliography{algCor}

\begin{thebibliography}{10}

\bibitem{bell2022invariant}
Jason Bell, Rahim Moosa, and Adam Topaz.
\newblock Invariant hypersurfaces.
\newblock {\em Journal of the Institute of Mathematics of Jussieu},
  21(2):713--739, 2022.

\bibitem{bellaiche2023self}
Jo{\"e}l Bella{\"\i}che.
\newblock On self-correspondences on curves.
\newblock {\em Algebra \& Number Theory}, 17(11):1867--1899, 2023.

\bibitem{hardouin}
Pierre Bonnet and Charlotte Hardouin.
\newblock Galoisian structure of large steps walks confined in the first
  quadrant.

\bibitem{buchacher2024separated}
Manfred Buchacher.
\newblock Separated variables on plane algebraic curves.
\newblock {\em arXiv e-prints}, pages arXiv--2411, 2024.

\bibitem{buchacher2024separating}
Manfred Buchacher.
\newblock Separating variables in bivariate polynomial ideals: the local case.
\newblock {\em arXiv preprint arXiv:2404.10377}, 2024.

\bibitem{buchacher2020separating}
Manfred Buchacher, Manuel Kauers, and Gleb Pogudin.
\newblock Separating variables in bivariate polynomial ideals.
\newblock In {\em Proceedings of the 45th International Symposium on Symbolic
  and Algebraic Computation}, pages 54--61, 2020.

\bibitem{fried1978poncelet}
Michael~D Fried.
\newblock Poncelet correspondences: Finite correspondences; {R}itt's theorem;
  and the {G}riffiths-{H}arris configuration for quadrics.
\newblock {\em Journal of Algebra}, 54(2):467--493, 1978.

\bibitem{fulton2008algebraic}
William Fulton.
\newblock Algebraic curves.
\newblock {\em An Introduction to Algebraic Geom}, 54, 2008.

\bibitem{hardouin2021differentially}
Charlotte Hardouin and Michael~F Singer.
\newblock On differentially algebraic generating series for walks in the
  quarter plane.
\newblock {\em Selecta Mathematica}, 27(5):89, 2021.

\bibitem{ingram2017critical}
Patrick Ingram.
\newblock Critical dynamics of variable-separated affine correspondences.
\newblock {\em Journal of the London Mathematical Society}, 95(3):1011--1034,
  2017.

\bibitem{ingram2018p}
Patrick Ingram.
\newblock p-adic uniformization and the action of galois on certain affine
  correspondences.
\newblock {\em Canadian Mathematical Bulletin}, 61(3):531--542, 2018.

\bibitem{ingram2019canonical}
Patrick Ingram.
\newblock Canonical heights for correspondences.
\newblock {\em Transactions of the American Mathematical Society},
  371(2):1003--1027, 2019.

\bibitem{krishnamoorthy2018correspondences}
Raju Krishnamoorthy.
\newblock Correspondences without a core.
\newblock {\em Algebra \& Number Theory}, 12(5):1173--1214, 2018.

\bibitem{shafarevich1994basic}
Igor~Rostislavovich Shafarevich and Miles Reid.
\newblock {\em Basic algebraic geometry}, volume~2.
\newblock Springer, 1994.

\bibitem{silverman2007arithmetic}
Joseph~H Silverman.
\newblock {\em The arithmetic of dynamical systems}, volume 241.
\newblock Springer Science \& Business Media, 2007.

\end{thebibliography}

\end{document}